\documentclass{article}

\usepackage[a4paper, total={6in, 8in}]{geometry}

\usepackage[utf8]{inputenc} 
\usepackage[T1]{fontenc}    
\usepackage{hyperref}       
\usepackage{url}            
\usepackage{booktabs}       
\usepackage{amsfonts}       
\usepackage{nicefrac}       
\usepackage{microtype}      

\usepackage{amsmath,amssymb,amsthm,xcolor}
\usepackage{array}
\usepackage{float}

\newtheorem{remark}{\textit{Remark}}[section]

\newtheorem{lemma}{\textit{Lemma}}[section]

\newtheorem{proposition}{\textit{Proposition}}[section]
\newtheorem{definition}{\textit{Definition}}[section]
\usepackage{graphicx}
\usepackage{subcaption}
\usepackage{mdframed}
\usepackage{lipsum}
\newmdtheoremenv{theo}{Theorem}

\title{A theory on the absence of spurious optimality}

\author{
  C. Josz , Y. Ouyang , R. Y. Zhang , J. Lavaei , S. Sojoudi
}

\begin{document}

\maketitle

\begin{abstract} 
We study the set of continuous functions that admit no spurious local optima (i.e. local minima that are not global minima) which we term \textit{global functions}. They satisfy various powerful properties for analyzing nonconvex and nonsmooth optimization problems. For instance, they satisfy a theorem akin to the fundamental uniform limit theorem in the analysis regarding continuous functions. Global functions are also endowed with useful properties regarding the composition of functions and change of variables. Using these new results, we show that a class of nonconvex and nonsmooth optimization problems arising in tensor decomposition applications are global functions. 
This is the first result concerning nonconvex methods for nonsmooth objective functions. 
Our result provides a theoretical guarantee for the widely-used $\ell_1$ norm to avoid outliers in nonconvex optimization.
\end{abstract}

\section{Introduction}
\label{sec:Introduction}

A recent branch of research in optimization and machine learning consists in proving that simple and practical algorithms can solve nonconvex optimization problems. Applications include, but are not limited to, neural networks \cite{mei2016,solta2017}, dictionary learning \cite{agarwal2014,agarwal2014bis}, deep learning \cite{lu2017,yun2018}, mixed linear regression \cite{yi2016,sedghi2016}, and phase retrieval \cite{vaswani2017,chen2018}.
In this paper, we focus our attention on matrix completion/sensing \cite{ge2016,ding2018,li2018} and tensor recovery/decomposition \cite{anandkumar2015,anandkumar2014,ge2016bis,jain2014}. Matrix completion/sensing aims to recover an unknown positive semidefinite matrix $M$ of known size $n$ and rank $r$ from a finite number of linear measurements modeled by the expression $\langle A_i , M \rangle := \text{trace}(A_i M), ~ i =1, \hdots,m,$ where the symmetric matrices $A_1,\hdots,A_m$ of size $n$ are known. It is assumed that the measurements contain noise which can modeled as $b_i := \langle A_i , M \rangle + \epsilon_i$ where $\epsilon_i$ is a realization of a random variable. When the noise is Gaussian, the maximum likelihood estimate of $M$ can be recast as the nonconvex optimization problem
\begin{equation}
\inf_{M \succcurlyeq 0} ~~~ \sum_{i=1}^m ~ \left( \langle A_i , M \rangle - b_i \right)^2 ~~~~ \text{subject to} ~~~~ \text{rank}(M) = r
\end{equation}
where $M \succcurlyeq 0$ stands for positive semidefinite.
One can remove the rank constraint and obtain a convex relaxation. It can then be solved via semidefinite programming after the reformulation of the objective function in a linear way. However, the computational complexity of the resulting problem is high, which makes it impractical for large-scale problems. A popular alternative is due to Burer and Monteiro \cite{burer2003,boumal2018}:
\begin{equation}
\inf_{X \in \mathbb{R}^{n \times r}} ~~~ \sum_{i=1}^m ~ \left( \langle A_i , X X^T \rangle - b_i \right)^2 
\label{eq:L2_general}
\end{equation}
This nonlinear \textit{Least-Squares} (LS) problem can be solved efficiently and on a large-scale with the Gauss-Newton method for instance. It has received a lot of attention recently due to the discovery in \cite{ge2016,bho2016} stating that the problem generally admits no spurious local minima (i.e. local minima that are not global minima). We raise the question of whether this also holds in the case of Laplacian noise, which is a better model to account for outliers in the data. The maximum likelihood estimate of $M$ can be converted to the \textit{Least-Absolute Value} (LAV)  optimization problem
\begin{equation}
\inf_{X \in \mathbb{R}^{n \times r}} ~~~ \sum_{i=1}^m ~ \left| \langle A_i , X X^T \rangle - b_i \right|.
\label{eq:L1_general}
\end{equation}
The nonlinear problem can be solved efficiently using nonconvex methods (for some recent work, see \cite{khamaru2018}).
For example, one may adopt the famous reformulation technique for converting $\ell_1$ norms to linear functions subject to linear inequalities to cast the above problem as a smooth nonconvex quadratically-constrained quadratic program \cite{boyd2009}. However, the analysis of this result has not been addressed in the literature - all ensuing papers (e.g. \cite{ge2017,zhu2018,balcan2017}) on matrix completion since the aforementioned discovery deal with smooth objective functions.

Consider $y \in \mathbb{R}^n$ and assume $r=1$. On the one hand, in the fully observable case with $M = y y^T$, the above nonconvex LS problem \eqref{eq:L2_general} consists in solving
\begin{equation}
\inf_{x \in \mathbb{R}^n}\sum\limits_{i,j=1}^n (x_i x_j - y_i y_j- \epsilon_{i,j})^2
\label{eq:L2noisy}
\end{equation}
for which there are no spurious local minima with high probability when $\epsilon_{i,j}$ are i.i.d. Gaussian variables \cite{ge2016}.
On the other hand, in the full observable case, the LAV problem \eqref{eq:L1_general} aims to solve
\begin{equation}
\inf_{x \in \mathbb{R}^n} \sum\limits_{i,j=1}^n |x_i x_j - y_i y_j- \epsilon_{i,j}|.
\label{eq:L1noisy}
\end{equation}

Although the LS problem has nice properties with Gaussian noise, we observe that stochastic gradient descent (SGD) fails to recover the matrix $M = y y^T$ in the presence of large but sparse noise. In contrast, SGD can perfectly recover the matrix by solving the LAV problem even when the sparse noise $\epsilon_{i,j}$ has a large amplitude.
Figures \ref{fig:n20} and \ref{fig:n50} show our experiments for $n = 20$ and $n=50$ with the number of noisy elements ranging from $0$ to $n^2$. See Appendix \ref{subsec:experiments} for our experiment settings.


\begin{figure}[!h]
\begin{subfigure}{.5\textwidth}
\centering
\includegraphics[width=6cm]{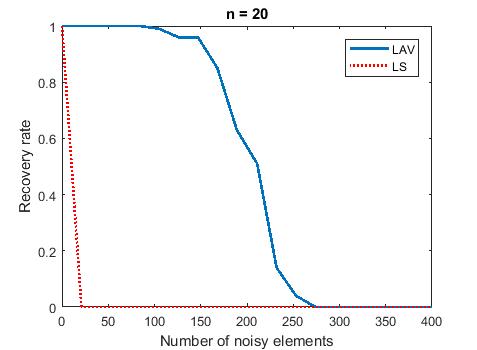}
	\caption{$n = 20$}
	\label{fig:n20}
\end{subfigure}
\begin{subfigure}{.5\textwidth}
\centering
\includegraphics[width=6cm]{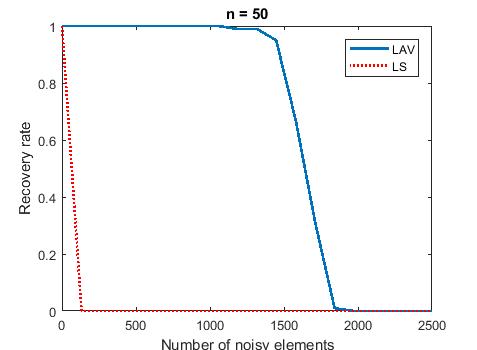}
	\caption{$n = 50$}
	\label{fig:n50}
\end{subfigure}
\caption{Experiments with sparse noise}
\end{figure}


Upon this LAV formulation hinges the potential of nonconvex methods to cope with sparse noise and with Laplacian noise. There is no result on the analysis of the local solutions of this nonsmooth problem in the literature even for the noiseless case. This could be due to the fact that the optimality conditions for the smooth reformulated version of this problem in the form of quadratically-constrained quadratic program are highly nonlinear and lead to an exponential number of scenarios. 
As such, the goal of this paper is to prove the following proposition, which as the reader will see, is a significant hurdle. It addresses the matrix noiseless case and more generally the case of a tensor of order $d \in \mathbb{N}$.
\begin{proposition}
\label{prop:L1_strong}
The function $f_{1}:  \mathbb{R}^n  \longrightarrow \mathbb{R} $ defined as 
\begin{equation}
f_{1}(x) := \sum\limits_{i_1,\hdots,i_d=1}^n |x_{i_1} \hdots x_{i_d} - y_{i_1} \hdots y_{i_d}|
\end{equation}
has no spurious local minima.
\end{proposition}

\begin{remark}
While Proposition \ref{prop:L1_strong} is true, its proof relies on Lemma \ref{lemma:clarke} whose proof is invalid. The
error lies in between \eqref{eq:before_error} and \eqref{eq:after_error} where an explanation has been inserted. Note that the proof of
the slightly weaker statement that $f_1$ has no spurious strict local minima (i.e., Proposition \ref{prop:L1}) is
valid. The proof attempt of Proposition \ref{prop:L1_strong} does provide useful insights however, in particular by
highlighting the role of a certain staircase function \eqref{eq:root}. A corrected proof in the matrix case $d=2$
building on this idea is object of the recent preprint \cite{josz2024}.
\end{remark}

A direct consequence of Proposition \ref{prop:L1_strong} is that one can perform the rank-one tensor decomposition by minimizing the function in Proposition \ref{prop:L1_strong} using a local search algorithm. Whenever the algorithm reaches a local minimum, it is a globally optimal solution leading to the desired decomposition.

Existing proof techniques, e.g. \cite{ge2017,ge2016,ding2018,li2018,anandkumar2015,anandkumar2014,ge2016bis,jain2014},
are not directly useful for the analysis of the nonconvex and nonsmooth optimization problem stated above. The Clarke derivative \cite{clarke1975,clarke1990} provides valuable insight (see Lemma \ref{lemma:clarke}) but it is not conclusive. In order to pursue the proof (see Lemma \ref{lemma:box}), we propose the new notion of global function. 
Unlike the previous approaches, it does not require one to exhibit a direction of descent. After some successive transformations, we reduce the problem to a linear program. It is then obvious that there are no spurious local minima. Incidentally, global functions provide a far simpler and shorter proof to a slightly weaker result, that is to say, the absence of spurious strict local minima. It eschews the Clarke derivative all together and instead considers a sequence of converging differentiable functions that have no spurious local minima (see Proposition \ref{prop:L1}). In fact, this technique also applies if we substitute the $\ell_1$ norm with the $\ell_\infty$ norm (see Proposition \ref{prop:max}). 

The paper is organized as follows. Global functions are examined in Section \ref{sec:Notion of global function} and their application to tensor decomposition is discussed in Section \ref{sec:Application to tensor decomposition}. Section \ref{sec:Conclusion} concludes our work. The proofs may be found in the supplementary material (Section \ref{sec:Appendix} of the supplementary material).

\section{Notion of global function}
\label{sec:Notion of global function}

Given an integer $n$, consider the Euclidian space $\mathbb{R}^n$ 
with norm $\|x\|_2 := \sqrt{\sum\limits_{i=1}^n x_i^2}$
along with a subset $S \subset \mathbb{R}^n$. The next two definitions are classical.

\begin{definition}
\label{def:global minimum}
We say that $x \in S$ is a global minimum of $f : S \longrightarrow \mathbb{R}$ if for all $y \in S \setminus \{x\}$, it holds that $f(x) \leqslant f(y)$. 
\end{definition}

\begin{definition}
\label{def:local minimum}
We say that $x \in S$ is a local minimum (respectively, strict local minimum) of $f : S \longrightarrow \mathbb{R}$ if there exists $\epsilon>0$ such that for all $y \in S \setminus \{x\}$ satisfying $\|x - y \|_2 \leqslant \epsilon$, it holds that $f(x) \leqslant f(y)$ (respectively, $f(x) < f(y)$). 
\end{definition}

We introduce the notion of global functions below. 

\begin{definition}
\label{def:global function}
We say that $f : S \longrightarrow \mathbb{R}$ is a \textit{global function} if it is continuous and its local minima are all global minima. Define $\mathcal{G}(S)$ as the set of all global functions on $S$.
\end{definition}

In the following, we compare global functions with other classes of functions in the literature, particularly those that seek to generalize convex functions.

When the domain $S$ is convex, two important proper subsets of $\mathcal{G}(S)$ are the sets of convex and strict quasiconvex functions. Convex functions (respectively, strict quasiconvex \cite{definetti1949,fenchel1951}) are such that $f(\lambda x + (1-\lambda)y) \leqslant \lambda f(x) + (1-\lambda)f(y)$ (respectively, $f(\lambda x + (1-\lambda)y) < \max\{ f(x) , f(y) \}$) for all $x,y \in S$ (with $x \neq y$) and $0 < \lambda < 1$. To see why these are proper subsets, notice that the cosinus function on $[0,4\pi]$ is a global function that is neither convex nor strict quasiconvex. In dimension one, global and strict quasiconvex functions are very closely related. Indeed, when the domain is convex and compact (i.e. an interval $[a,b]$ where $a,b \in \mathbb{R}$), it can be shown that a function is strict quasiconvex if and only if it is global and has a unique global minimum. However, this is not true in higher dimensions, as can be seen in Figure \ref{fig:example} in Appendix \ref{subsec:eg}, or in the existing literature, i.e. in \cite{dunn1987} or in \cite[Figure 1.1.10]{bertsekas2016}. It is also not true in dimension one if we remove the assumption that the domain is compact (consider $f(x) := (x^2 + x^4) / (1+x^4)$ defined on $\mathbb{R}$ and illustrated in Figure \ref{fig:example} in Appendix \ref{subsec:eg}). 

When the domain $S$ is not necessarily convex, a proper subset of $\mathcal{G}(S)$ is the set of star-convex functions. 
For a star-convex function $f$, there exists $x \in S$ such that $f(\lambda x + (1-\lambda)y) \leqslant \lambda f(x) + (1-\lambda)f(y)$ for all $y \in S \setminus \{x\}$ and $0 < \lambda < 1$. Again, the cosinus function on $[0,4\pi]$ is a global function that is not star-convex. Another interesting proper subset of $\mathcal{G}(S)$ is the set of functions for which, informally, given any point, there exists a strictly strictly decreasing path from that point to a global minimum. This property is discussed in \cite[P.1]{venturi2018} (see also \cite{freeman2017}) to study the landscape of loss functions of neural networks. Formally, the property is that for all $x \in S$ such that $f(x) > \inf_{y \in S} f(y)$, there exists a continuous function $g: [0,1] \longrightarrow S$ such that $g(0) = x$, $g(1) \in \text{argmin} \{ f(y) ~|~  y \in S \}$, and $t \in [0,1] \longmapsto f(g(t))$ is strictly decreasing (i.e. $f(g(t_1)) > f(g(t_2))$ if $0 \leqslant t_1<t_2 \leqslant 1$). Not all global functions satisfy this property, as illustrated by the function in Figure \ref{fig:example}. For instance, there exists no strictly decreasing path from $x=-3$ to the global minimizer $0$. However, in the funtion in Figure \ref{fig:global} in Appendix \ref{subsec:eg}, there exists a strictly decreasing path from any point to the unique global minimizer. One could thus think that if $S$ is compact, or if $f$ is coercive, then one should always be able to find a strictly decreasing path. However, there need not exist a strictly decreasing path in general. Consider for example the function defined on $([-1,1]\setminus\{0\}) \times [-1,1]$ as follows
$$
f(x_1,x_2) ~ := ~ 
\left\{
\begin{array}{cl}
-4|x_1|^3(1-x_2)\left(\sin\left(-\frac{1}{|x_1|}\right)+1\right) & \text{if}~ \hphantom{....}0 \leqslant x_2 \leqslant 1, \\[.4cm]
\left\{ 12 |x_1|^3\left(\sin\left(-\frac{1}{|x_1|}\right)+1\right) - 2 \right\} x_2^3 ~+ &  \\[.4cm] 
\left\{ 20 |x_1|^3\left(\sin\left(-\frac{1}{|x_1|}\right)+1\right) - 3 \right\} x_2^2 ~+ & \text{if}~ -1 \leqslant x_2 < 0.\\[.4cm] 
4|x_1|^3\left(\sin\left(-\frac{1}{|x_1|}\right)+1\right) x_2 - 4|x_1|^3\left(\sin\left(-\frac{1}{|x_1|}\right)+1\right) &
\end{array}
\right.
$$

The function and its differential can readily be extended continuously to $[-1,1] \times [-1,1]$. This is illustrated in Figure \ref{fig:nopath} in Appendix \ref{subsec:eg}. This yields a smooth\footnote{In fact, one could make it infinitely differentiable by using the exponential function in the construction, but it is more cumbersome.} 
global function for which there exists no strictly decreasing path from the point $x=(0,1/2)$ to a global minimizer (i.e. any point in $[-1,1] \times \{ -1 \}$). We find this to be rather counter-intuitive. To the best of our knowledge, no such function has been presented in past literature. Hestenes \cite{hestenes1975} considered the function defined on $[-1,1] \times [-1,1]$ by $f(x_1,x_2) := (x_2-x_1^2)(x_2-4x_1^2)$ (see also \cite[Figure 1.1.18]{bertsekas2016}). It is a global function for which the point $x = (0,0)$ (which is not a global minimizer) admits no direction of descent, i.e. $d \in \mathbb{R}^2$ such that $t \in [0,1] \longmapsto f(x+td)$ is strictly decreasing. However, it does admit a strictly decreasing path to a global minimizer, i.e. $t \in [0,1] \longmapsto (\frac{\sqrt{10}}{4}t,t^2)$, along which the objective equals $-\frac{9}{16} t^4$. This is unlike the function exhibited in Figure \ref{fig:nopath}. As a byproduct, our function shows that the generalization of quasiconvexity to non-convex domains described in \cite[Chapter 9]{avriel1988} is a proper subset of global functions. This generalization was proposed in \cite{ortega1970} and further investigated in \cite{avriel1980,horst1982,horst1988,burai2013}. It consists in replacing the segment used to define convexity and quasiconvexity by a continuous path. 


Finally, we note that there exists a characterization of functions whose local minima are global, without requiring continuity as in global functions. It is based on a certain notion of continuity of sublevel sets, namely lower-semicontinuity of point-to-set mappings \cite[Theorem 3.3]{zang1975}. We will see below that continuity is a key ingredient for obtaining our results. We do not require more regularity precisely because one of our goals is to study nonsmooth functions. Speaking of which, observe that global functions can be nowhere differentiable, contrary to convex functions \cite[Theorems 2.1.2 and 2.5.1]{borwein2010}. Consider for example the global function defined on $]0,1[ ~ \times ~ ]0,1[$ by
$
f(x_1,x_2) := |2x_2-1| \sum_{n=0}^{+\infty} s(2^n x_1)/2^n
$
where $s(x) := \min_{ n \in \mathbb{N} } |x - n|$ is the distance to nearest integer. For any fixed $x_2 \neq 0$, the function $x_1 \in [0,1] \longmapsto f(x_1,x_2)/|x_2|$ is the Takagi curve \cite{takagi1901,allaart2011,lagarias2012} which is nowhere differentiable. It can easily be deduced that the bivariate function is nowhere differentiable. This is illustrated in Figure \ref{fig:nodiff}.

In the following, we review some of the properties of global functions. Their proofs can be found in the appendix. We begin by investigating the composition operation.

\begin{proposition}[Composition of functions]
\label{prop:composition}
Consider $f : S \longrightarrow \mathbb{R}$. Let $\phi: f(S) \longrightarrow \mathbb{R}$ denote a strictly increasing function where $f(S)$ is the range of $f$. It holds that $f \in \mathcal{G}(S)$ if and only if $\phi \circ f \in \mathcal{G}(S)$.
\end{proposition}

However, the set of global functions is not closed under composition of functions in general. For instance, $f(x):= |x|$ and $g(x) := \max( -1 , |x|-2)$ are global functions on $\mathbb{R}$, but $f \circ g$ is not global function on $\mathbb{R}$.

\begin{proposition}[Change of variables]
\label{prop:change}
Consider $f : S \longrightarrow \mathbb{R}$, a subset $S' \subset \mathbb{R}^n$, and a homeomorphism $\varphi: S \longrightarrow S'$. It holds that $f \in \mathcal{G}(S)$ if and only if $f \circ \varphi^{-1} \in \mathcal{G}(S')$.
\end{proposition}

Next, we consider what happens if we have a sequence of global functions. 
Figure \ref{fig:pointwise} shows that the sequence of global functions (red dotted curves) pointwise converges to a function with a spurious local minimum (blue curve).
Figure \ref{fig:uniform} shows that uniform convergence also does not preserve the property of being a global function: all points on the middle part of the limit function (blue curve) are spurious local minima. However, it suggests that uniform convergence preserves a slightly weaker property than being a global function. Intuitively, the limit should behave like a global function except that it may have ``flat'' parts. We next formalize this intuition. To do so, we consider the notions of global minimum, local minimum, and strict local minimum  (Definition \ref{def:global minimum} and Definition \ref{def:local minimum}), which apply to points in $\mathbb{R}^n$, and generalize them to subsets of $\mathbb{R}^n$. We will borrow the notion of neighborhood of a set (\textit{uniform neighborhood} to be precise).


\begin{figure}[!h]
\begin{subfigure}{.5\textwidth}
\centering
\includegraphics[width=7cm]{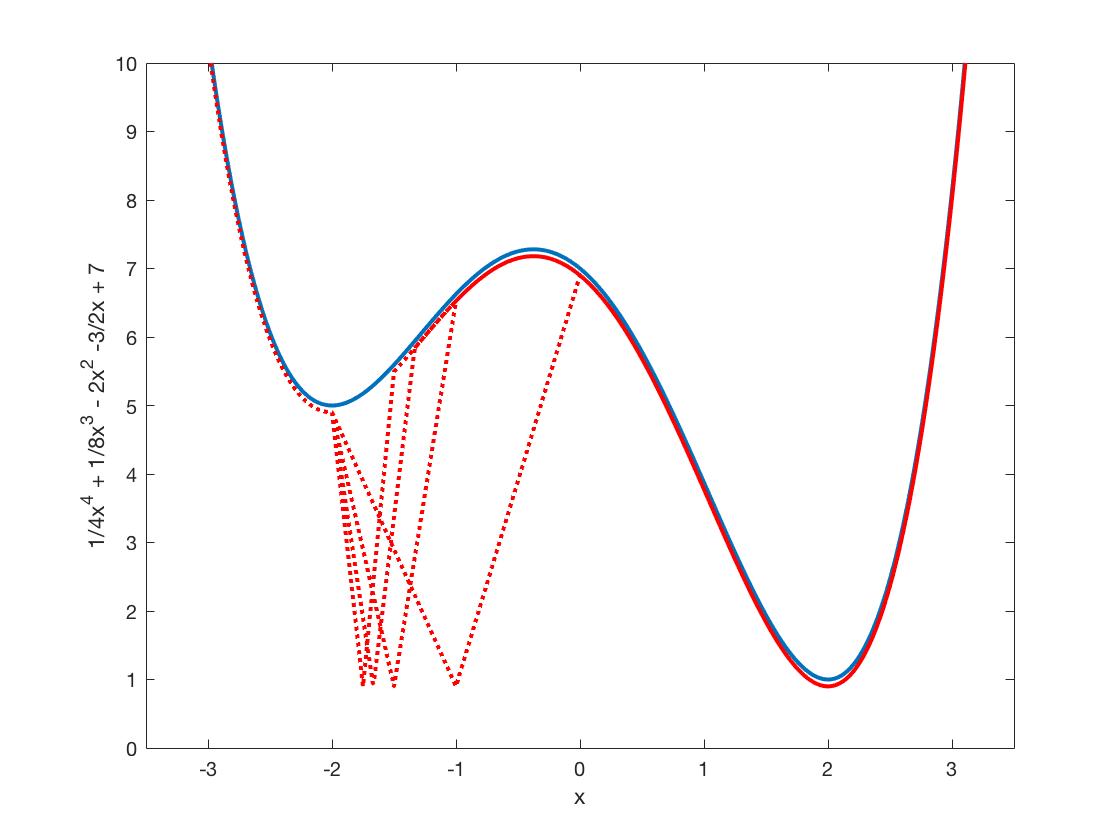}
\caption{Pointwise convergence}
	 \label{fig:pointwise}
\end{subfigure}
\begin{subfigure}{.5\textwidth}
\centering
\includegraphics[width=7cm]{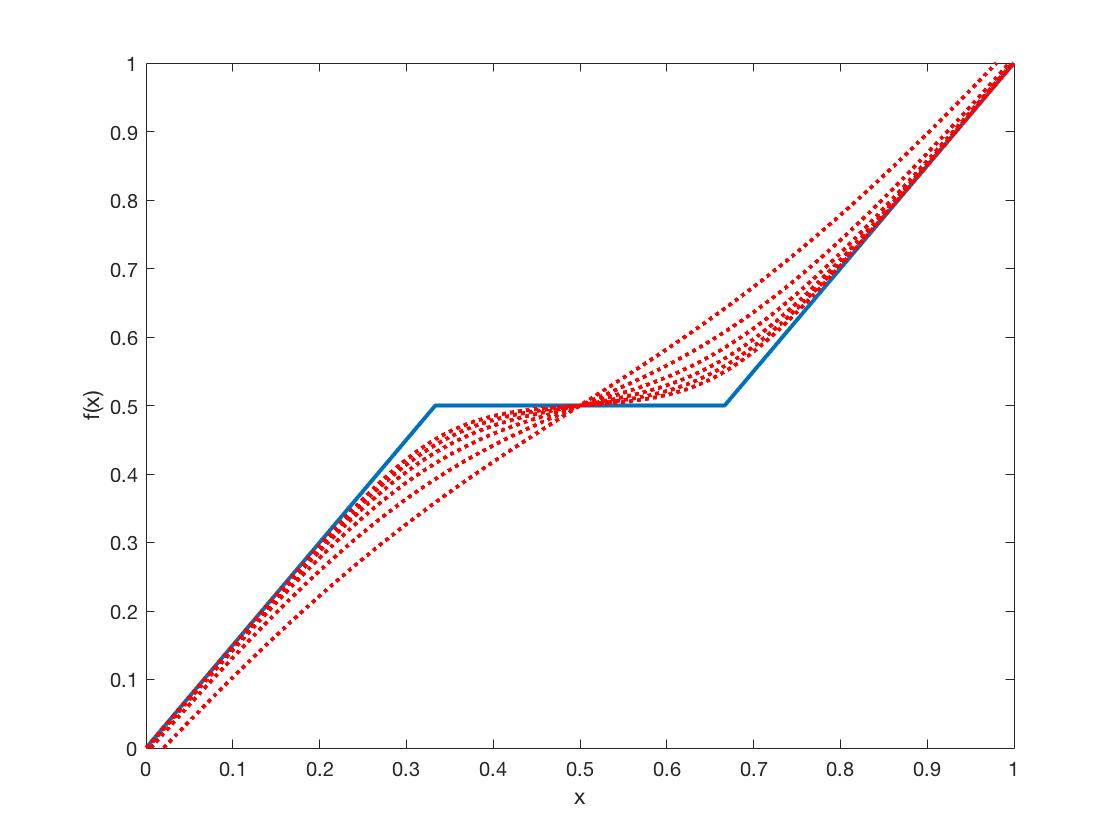}
\caption{Uniform convergence}
	 \label{fig:uniform}
\end{subfigure}
\caption{Convergence of a sequence of global functions}
\label{fig:convergence}
\end{figure}


\begin{definition} We say that a subset $X \subset S$ is a global minimum of $f : S \longrightarrow \mathbb{R}$ if $\inf_{X} f \leqslant \inf_{S \setminus X} f$. 
\end{definition}

We note in passing the following two propositions. We will use them repeatedly in the next section. The proofs are omitted as they follow directly from the definitions.

\begin{proposition}
\label{prop:subset}
Assume that the following statements are true:
\begin{enumerate}
\item $X \subset S$ is a global minimum of $f$;
\item $f \in \mathcal{G}(X)$;
\item $f$ does not have any local minima on $S \setminus X$.
\end{enumerate}
Then, $f \in \mathcal{G}(S)$.
\end{proposition}

\begin{proposition}
\label{prop:cover}
If $(X_\alpha)_{\alpha \in A}$ are global minima for some index set $A$, then 
$
\bigcap_{\alpha \in A} \mathcal{G}( X_\alpha ) ~ \subset ~ \mathcal{G}\left( \bigcup_{\alpha \in A} X_\alpha \right).
$
\end{proposition}

We proceed to generalize the definition of local minimum.

\begin{definition} We say that a compact subset $X \subset S$ is local minimum (respectively, strict local minimum) of $f : S \longrightarrow \mathbb{R}$ if there exists $\epsilon>0$ such that for all $x \in X$ and for all $y \in S \setminus X$ satisfying $\|x - y \|_2 \leqslant \epsilon$, it holds that $f(x) \leqslant f(y)$ (respectively, $f(x) < f(y)$).\footnote{Note that the neighborhood of a compact set is always uniform.}
\end{definition}

The above definitions are distinct from the notion of valley proposed in \cite[Definition 1]{venturi2018}. The latter is defined as a connected component\footnote{A subset $C \subset S$ is connected if it is not equal to the union of two disjoint nonempty closed subsets of $S$. A maximal connected subset (ordered by inclusion) of $S$ is called a connected component.} of a sublevel set (i.e. $\{ x \in S ~|~ f(x) \leqslant \alpha \}$ for some $\alpha \in \mathbb{R}$). Local minima and strict local minima need not be valleys, and vice-versa. 
One may easily check that when the set is a point, i.e. $X = \{x\}$ with $x \in S$, the two definitions above are the same as the previous definitions of minimum (Definition \ref{def:global minimum} and Definition \ref{def:local minimum}). They are therefore consistent. It turns out that the notion of global function (Definition \ref{def:global function}) does not change when we interpret it in the sense of sets. We next verify this claim.

\begin{proposition}[Consistency of Definition \ref{def:global function}]
\label{prop:definition}
Let $f : S \longrightarrow \mathbb{R}$ denote a continuous function. All local minima are global minima in the sense of points if only if all local minima are global minima in the sense of sets.
\end{proposition}

We are ready to define a slightly weaker notion than being a global function. 

\begin{definition}
\label{def:weakly global}
We say that $f : S \longrightarrow \mathbb{R}$ is a \textit{weakly global function} if it is continuous and if all strict local minima are global minima in the sense of sets.
\end{definition}

The generalization from points to sets in the definition of a minimum is justified here, as can be seen in Figure \ref{fig:weak}. All strict local minima are global minima in the sense of points. However, $X=[a,b]$ with $a\approx -2.6$ and $b = -1$ is a strict local minimum that is not a global minimum. Indeed, $\inf_X f = 6 > 1 = \inf_{\mathbb{R} \setminus X} f$. Hence, the function is not weakly global.


\begin{figure}[!h]
	\centering
	\includegraphics[width=.5\textwidth]{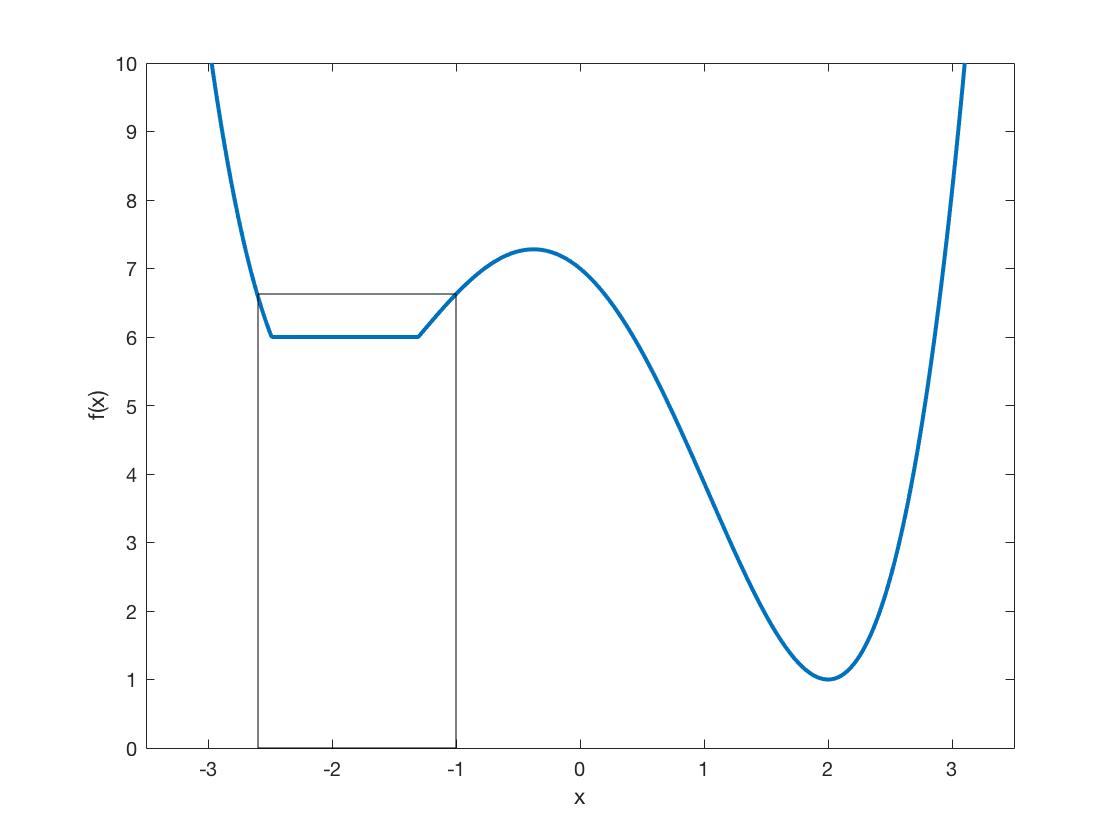}
	\caption{All strict local minima are global minima in the sense of points but not in the sense of sets.}
	 \label{fig:weak}
\end{figure}


We next make the link with the intuition regarding the flat part in Figure \ref{fig:uniform}.

\begin{proposition}
\label{prop:weakly global}
If $f : S \longrightarrow \mathbb{R}$ is a weakly global function, then it is constant on all local minima that are not global minima.
\end{proposition}

We are interested in functions that are potentially defined on all of $\mathbb{R}^n$ (i.e. unconstrained optimization) or on subsets $S \subset \mathbb{R}^n$ that are not necessarily compact (i.e. general constrained optimization). We therefore need to borrow a slightly more general notion than uniform convergence \cite[page 95, Section 3]{remmert1991}.

\begin{definition} We say that a sequence of continuous functions $f_k : S \longrightarrow \mathbb{R}, k = 1,2,\ldots,$ converges \textit{compactly} towards $f: S \longrightarrow \mathbb{R}$ if for all compact subsets $K \subset S$, the restrictions of $f_k$ to $K$ converge uniformly towards the restriction of $f$ to $K$.
\end{definition}

We are now ready to state a result regarding the convergence of a sequence of global functions and an important property that is preserved in the process.

\begin{proposition}[Compact convergence]
\label{prop:compact}
Consider a sequence of functions $(f_k)_{k \in \mathbb{N}}$ and a function $f$, all from $S \subset \mathbb{R}^n$ to $\mathbb{R}$. If 
\begin{equation}
f_k \longrightarrow f ~~~ \text{compactly} 
\end{equation}
and if $f_k$ are global functions on $S$, then $f$ is a weakly global function on $S$.
\end{proposition}

Note that the proofs in this section are not valid if we replace the Euclidian space by an infinite-dimensional metric space. Indeed, we have implicitely used the fact that the unit ball is compact in order for the uniform neighborhood of a minimum to be compact.

\section{Application to tensor decomposition}
\label{sec:Application to tensor decomposition}

Global functions can be used to prove the following two significant results on nonlinear functions involving $\ell_1$ norm and $\ell_\infty$ norm, as explained below.

\begin{proposition}
\label{prop:L1}
The function $f_{1}:  \mathbb{R}^n  \longrightarrow  \mathbb{R} $ defined as 
\begin{equation}
f_{1}(x) := \sum\limits_{i_1,\hdots,i_d=1}^n |x_{i_1} \hdots x_{i_d} - y_{i_1} \hdots y_{i_d}|
\end{equation}
is a weakly global function; in particular, it has no spurious strict local minima.
\end{proposition}
\begin{proof}
The functions
\begin{equation}
\label{eq:f_k}
f_p(x) := \sum\limits_{i_1,\hdots,i_d=1}^n |x_{i_1} \hdots x_{i_d} - y_{i_1} \hdots y_{i_d}|^{p}
\end{equation}
for $p \longrightarrow 1$ with $p>1$
form a set of global functions that converge compactly towards the function $f_{1}$. This is illustrated in Figure \ref{fig:uniform matrix} in Appendix \ref{subsec:eg} for $n=d=2$ and $y = ( 1 , -3/4)$. The desired result then follows from Proposition \ref{prop:compact}. To see why each $f_p$ is a global function, observe that $f_p$ is differentiable with the first-order optimality condition as follows:
$$
\sum\limits_{i_1,\hdots,i_{d-1}=1}^n x_{i_1} \hdots x_{i_{d-1}} ( x_{i_1} \hdots x_{i_{d-1}}  x_i - y_{i_1} \hdots y_{i_{d-1}}y_i ) | x_{i_1} \hdots x_{i_{d-1}}  x_i - y_{i_1} \hdots y_{i_{d-1}}y_i |^{p-2} = 0
$$
for all $i \in \{1,\dots,n\}$. Note that each term in the sum converges towards zero if the expression inside the absolute value converges towards zero, so that the equation in well-defined. Consider a local minimum $x \in \mathbb{R}^n$; then, $x$ must satisfy the above first-order optimality condition.
If $y_i = 0$, then the above equation readily yields $x_i = 0$. This reduces the problem dimension from $n$ variables to $n-1$ variables, so without loss of generality we may assume that $y_i \neq 0,~i =1,\hdots,m$. After a division, observe that the following equation is satisfied
$$
\sum\limits_{i_1,\hdots,i_{d-1}=1}^n |y_{i_1} \hdots y_{i_{d-1}}|^{p} ~ \frac{x_{i_1} \hdots x_{i_{d-1}}}{y_{i_1} \hdots y_{i_{d-1}}} \left( \frac{x_{i_1} \hdots x_{i_{d-1}}}{y_{i_1} \hdots y_{i_{d-1}}} t - 1 \right) \left| \frac{x_{i_1} \hdots x_{i_{d-1}}}{y_{i_1} \hdots y_{i_{d-1}}}t - 1 \right|^{p-2} = 0
$$
for all $t \in \{ x_1/y_1, \hdots, x_n / y_n\}$. Each term with $x_{i_1} \hdots x_{i_{d-1}} \neq 0$ in the above sum is a strictly increasing function of $t \in \mathbb{R}$ since it is the derivative of the strictly convex function 
\begin{equation}
g(t) = | x_{i_1} \hdots x_{i_{d-1}} t - y_{i_1} \hdots y_{i_{d-1}} |^{p}.
\end{equation}
The point $x=0$ is not a local minimum ($y$ is a direction of descent of $f_p$ at $0$), and thus $x \neq 0$. As a result, the above sum is a strictly increasing function of $t \in \mathbb{R}$. Hence, it has at most one root, that is to say $t = x_1/y_1 = \cdots = x_n / y_n$. Plugging in, we find that $t^d = 1$. If $d$ is odd, then $x=y$ and if $d$ is even, then $x=\pm y$. To conclude, any local minimum $x$ is a global minimum of $f_p$.
\end{proof}

\begin{proposition}
\label{prop:max}
$f_{\infty}:  \mathbb{R}^n  \longrightarrow  \mathbb{R} $ defined as 
\begin{equation}
f_{\infty}(x) := \max\limits_{1\leqslant i_1,\hdots,i_d \leqslant n} ~ |x_{i_1} \hdots x_{i_d} - y_{i_1} \hdots y_{i_d}|
\end{equation}
is a weakly global function; in particular, it has no spurious strict local minima.
\end{proposition}
\begin{proof}
The functions
$
h_p(x) := \left(\sum\limits_{i_1,\hdots,i_d=1}^n |x_{i_1} \hdots x_{i_d} - y_{i_1} \hdots y_{i_d}|^{p}\right)^{\frac{1}{p}}
$
for $p \longrightarrow + \infty$
form a set of global functions that converge compactly towards the function $f_{\infty}$. We know that each $h_p$ is a global function by applying Proposition \ref{prop:composition} to \eqref{eq:f_k} with the fact that $(\cdot)^{\frac{1}{p}}$ is increasing for nonnegative arguments.
\end{proof}

Note that the functions in Proposition \ref{prop:L1} and Proposition \ref{prop:max} are \textit{a priori} utterly different, yet both proofs are essentially the same. This highlights the usefulness of the new notion of global functions. 
\begin{remark}
The notion of weakly global functions explains that one can perform tensor decomposition by minimizing the nonconvex and nonsmooth functions in Proposition \ref{prop:L1} and Proposition \ref{prop:max} with a local search algorithm. Whenever the algorithm reports a strict local minimum, it is a globally optimal solution. 
\end{remark}

In order to strengthen the conclusion in Proposition \ref{prop:L1} and to establish the absence of spurious local minima, we propose the following two lemmas. 
Using Proposition \ref{prop:subset} and these two lemmas, we arrive at the stronger result stated in Proposition \ref{prop:L1_strong}. 

\begin{lemma}
\label{lemma:clarke}
If $x \in \mathbb{R}^n$ is a first-order stationary point of $f_{1}$
in the sense of the Clarke derivative, then the following statements hold:
\begin{enumerate}
\item If $y_i = 0$ for some $i \in \{1,\ldots, n\}$, then $x_i = 0$;
\item For all $i_1,\hdots,i_d \in \{1,\ldots, n\}$, it holds that
$\frac{x_{i_1} \hdots x_{i_{d}}}{y_{i_1} \hdots y_{i_{d}}} \leqslant 1.$
\end{enumerate}
\end{lemma}
\begin{proof}
Similar in spirit to the proof of Proposition \ref{prop:L1}, the ratios $t \in \{ x_1/y_1, \hdots, x_n / y_n\}$ for a first-order stationary point must all be the roots of an increasing (set-valued) ``staircase function".
We then obtain the results by analyzing the relation between the roots and the jump points of the staircase function.
See Appendix \ref{subsec:clarke} for the complete proof.
\end{proof}
Note that the above lemma only uses the first-order optimality condition (in the sense of Clarke derivative) without any direction of decent.

\begin{remark}
One cannot show that there are no spurious local minima with only the first-order optimality condition (in the Clarke derivative sense). In fact, any $x \in \mathbb{R}^n$ satisfying $\sum\limits_{i=1}^n |y_i | \frac{x_i}{y_i} = 0$ and 
$\frac{x_{i_1} \hdots x_{i_{d}}}{y_{i_1} \hdots y_{i_{d}}} \leqslant 1$ for all $i_1,\hdots,i_d \in \{1,\ldots, n\}$, is 
a first-order stationary point, but is not a local minimum.
\end{remark}

\begin{lemma}
\label{lemma:box} 
If $y_1 \hdots y_n \neq 0$, define the set
\begin{equation}
S := \left\{ x \in \mathbb{R}^n ~ \left| ~ \frac{x_{i_1} \hdots x_{i_{d}}}{y_{i_1} \hdots y_{i_{d}}} \leqslant 1~, ~~~~~~ \forall ~ i_1,\hdots,i_{d} \in \{1,\ldots, n\} \right. \right\} .
\end{equation}
Then, $f_{1} \in \mathcal{G}(S)$.
\end{lemma}
\begin{proof}
We provide a sketch here, and the complete proof is deferred to Appendix \ref{subsec:box}.
The objective function on $S$ is equal to $f_{1}(x) = \left( \sum\limits_{i=1}^n |y_i | \right)^d - \left( \sum\limits_{i=1}^n |y_i | \frac{x_i}{y_i} \right)^d$. 
Define the set $S' := \left\{ ~ x \in \mathbb{R}^n ~ \left| ~ x_{i_1} \hdots x_{i_d} \leqslant 1~, ~~~ \forall ~ i_1,\hdots,i_{d} \in \{1,\ldots, n\} \right. ~ \right\}$.
When $d$ is an odd number, the composition and change of variables properties of global functions (Propositions \ref{prop:composition} and \ref{prop:change}) imply that $f_{1}$ is a global function on $S$ if and only if $f_{\text{odd}}(x) = - \sum_{i=1}^n |y_i| x_i \in \mathcal{G}(S')$.
Similarly, when $d$ is an even number, $f$ is a global function if and only if $f_{\text{even}}(x) = - \left(\sum_{i=1}^n |y_i| x_i\right)^2 \in \mathcal{G}(S')$.
For the case when $d$ is odd, we apply the Karush-Kuhn-Tucker conditions to restrict attention to the positive orthant and conclude by showing its association with a linear program. 
For the case when $d$ is even, we divide the set $S'$ into two subsets: $S' \cap \{x| \sum_{i=1}^n |y_i| x_i \geq 0\}$ and $S' \cap \{x| \sum_{i=1}^n |y_i| x_i \leq 0\}$.
Observe that $f_{\text{even}}(x)$ is a global function on each of the subset by associating each subset with a linear program. 
Then, Proposition \ref{prop:subset} establishes the result.
\end{proof}

The two previous lemmas prove Proposition \ref{prop:L1_strong}; the notion of global function is used to the prove the latter.

\section{Conclusion}
\label{sec:Conclusion}
Nonconvex optimization appears in many applications, such as matrix completion/sensing, tensor recovery/decomposition, and training of neural networks. 
For a general nonconvex function, a local search algorithm may become stuck at a local minimum that is arbitrarily worse than a global minimum.
We develop a new notion of global functions for which all local minima are global minima.
Using certain properties of global functions, we show that the set of these functions include a class of nonconvex and nonsmooth functions that arise in matrix completion/sensing and tensor recovery/decomposition with Laplacian noise.
This paper offers a new mathematical technique for the analysis of nonconvex and nonsmooth functions such as those involving $\ell_1$ norm and $\ell_\infty$ norm.

\newpage
\bibliography{mybib}{}
\bibliographystyle{plain}

\section*{Acknowledgements}
Many thanks to Chris Dock for fruitful discussions.

\newpage
\section{Appendix}
\label{sec:Appendix}

\subsection{Experiment settings}
\label{subsec:experiments}

We use SGD to solve the problems \eqref{eq:L2noisy} and \eqref{eq:L1noisy} for randomly generated rank-one matrices.
In the experiments, each $y$ is generated according to the $n$-dimensional i.i.d. standard Gaussian distribution.
The positions of the sparse noise are uniformly selected from all the $n^2$ elements, and each noisy element is replaced by a Gaussian random variable with standard deviation $10$.
With regard to SGD, we set the learning rate to $0.001$ and momentum to $0.9$. The initial point is a Gaussian random vector.

In our experiments, a successful recovery means that the solution $x$ has a relative error less than $0.1$ compared with the optimal solution $y$.
We consider $n = 20$ and $n=50$ and vary the number of noisy elements from $0$ to $n^2$. For each case, we run $100$ experiments and report the successful recovery rate.
As shown in Figures \eqref{fig:n20} and \eqref{fig:n50}, the LS problem \eqref{eq:L2noisy} fails to recover the matrix except for the noiseless case. On the other hand, the LAV problem \eqref{eq:L2noisy} provides perfect recovery in the presence of sparse noise.

\subsection{Illustrations}
\label{subsec:eg}

This section is composed of Figure \ref{fig:uni}, Figure \ref{fig:uniform matrix}, and Figure \ref{fig:counter}.


\begin{figure}[!h]
\begin{subfigure}{.4\textwidth}
\centering
\includegraphics[width=6cm]{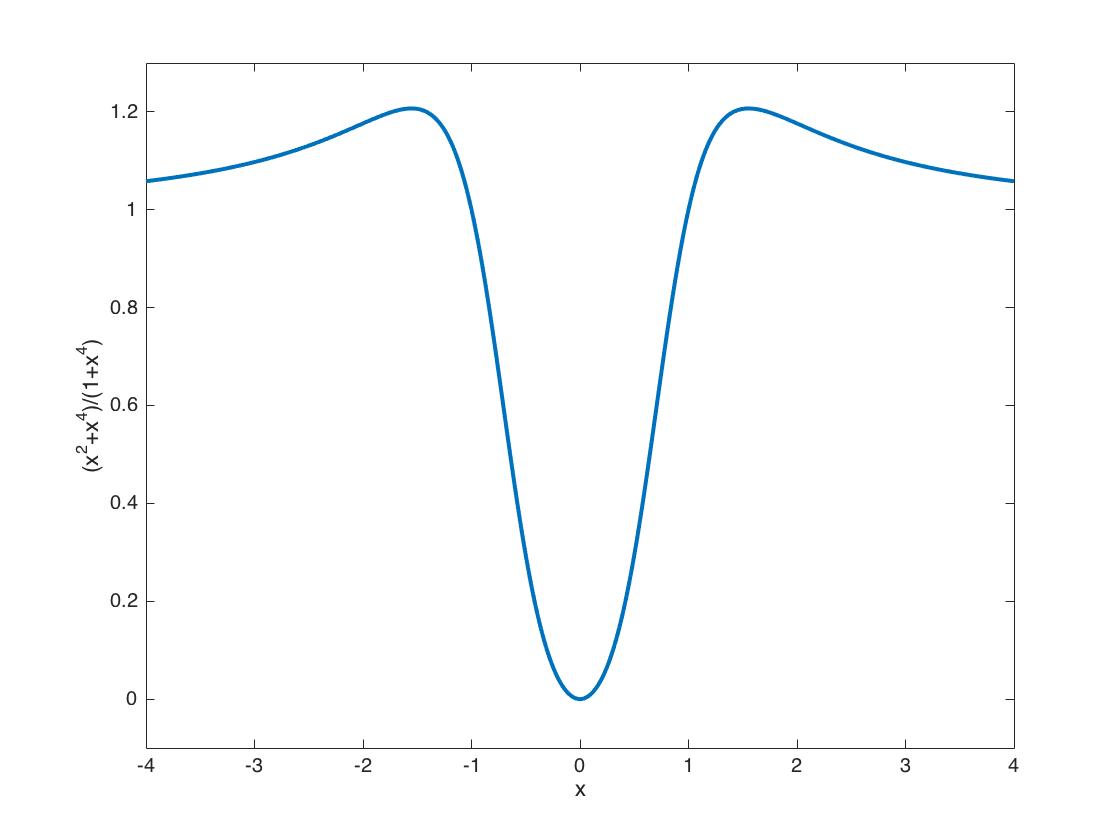}
	\caption{Global function on $\mathbb{R}$}
	 \label{fig:example}
\end{subfigure}
\begin{subfigure}{.9\textwidth}
\centering
\includegraphics[width=7cm]{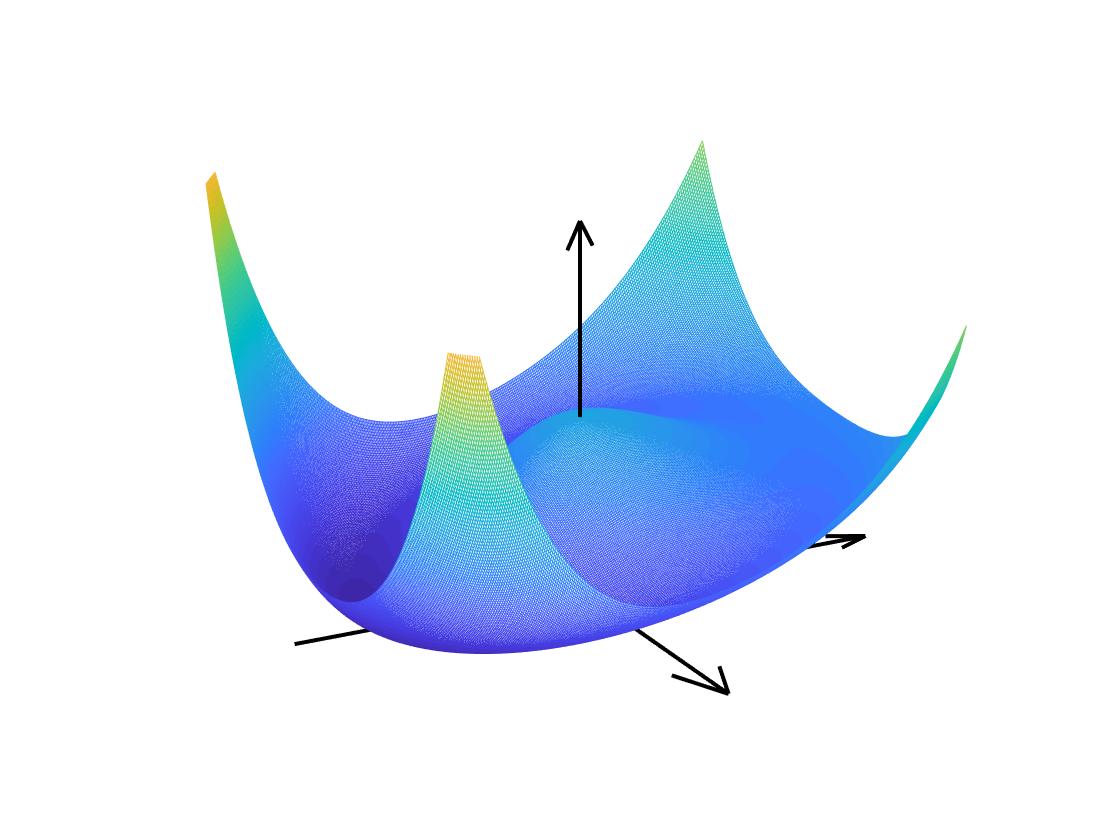}
	\caption{Global function on $\mathbb{R}^2$}
	 \label{fig:global}
\end{subfigure}
\caption{Examples of global functions}
\label{fig:uni}
\end{figure}

\begin{figure}[!h]
\begin{subfigure}{.4\textwidth}
\centering
\includegraphics[width=8cm]{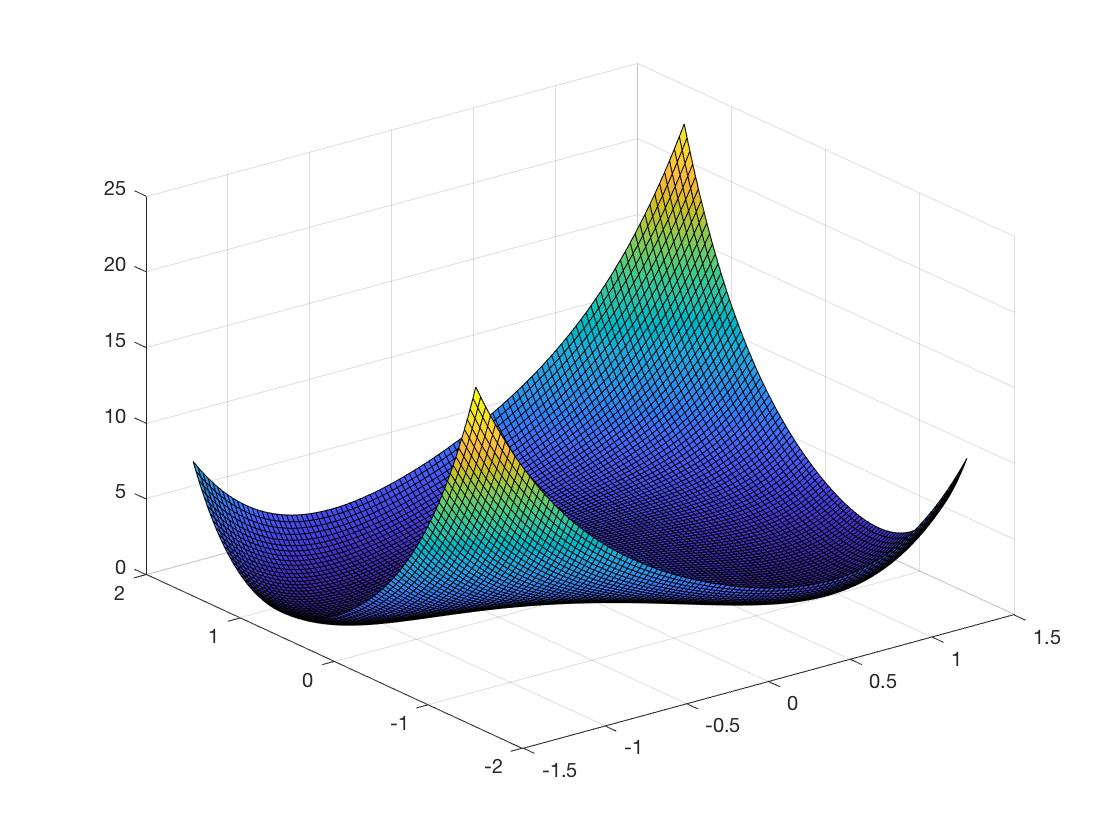}
\caption{$f_2$}
\end{subfigure}
\begin{subfigure}{.9\textwidth}
\centering
\includegraphics[width=8cm]{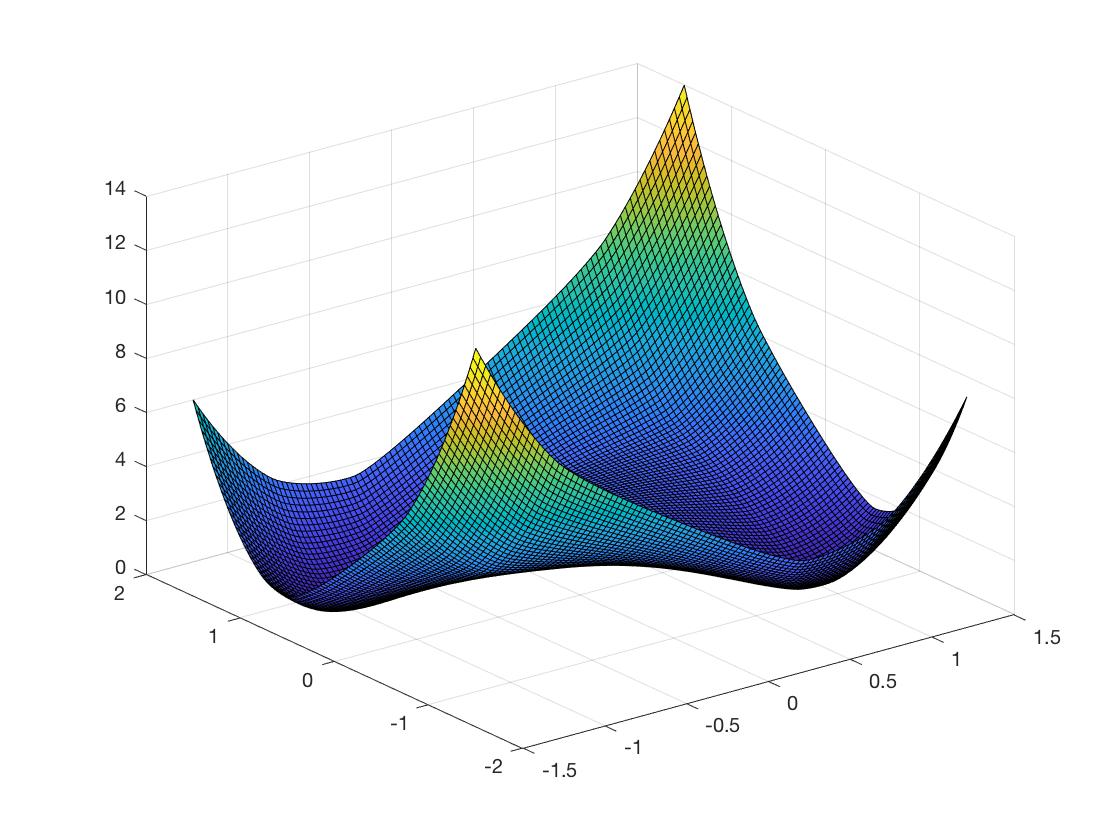}
\caption{$f_{1.5}$}
\end{subfigure}\par\medskip
\begin{subfigure}{.4\textwidth}
\centering
\includegraphics[width=8cm]{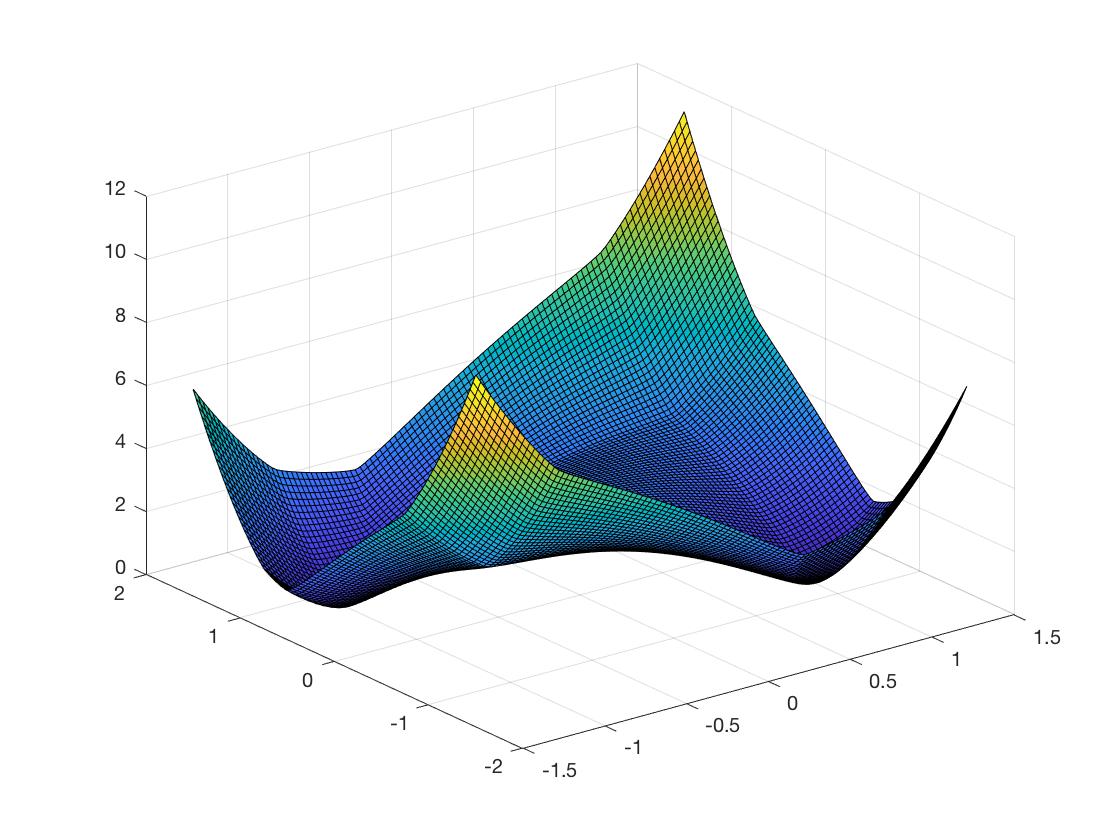}
\caption{$f_{1.25}$}
\end{subfigure}
\begin{subfigure}{.9\textwidth}
\centering
\includegraphics[width=8cm]{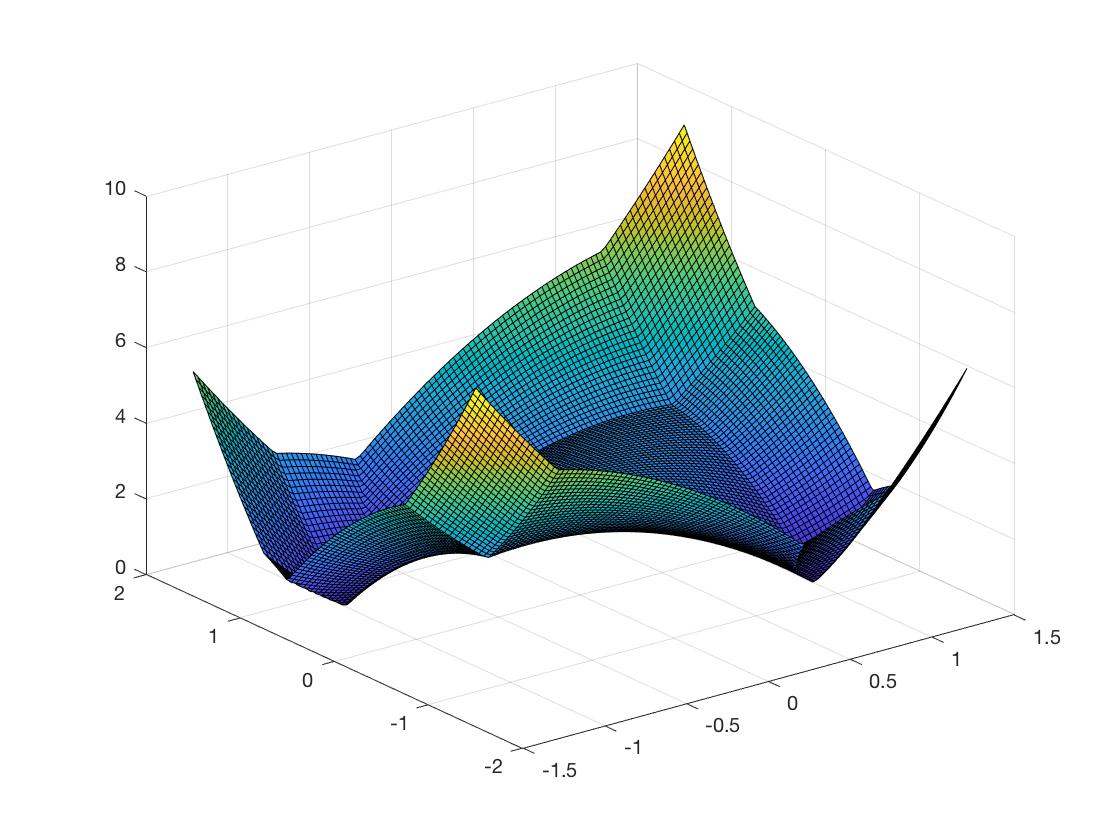}
\caption{$f_1$}
\end{subfigure}
\caption{Compact convergence of global functions implies that strict local minima are global}
\label{fig:uniform matrix}
\end{figure}

\begin{figure}[!h]
\begin{subfigure}{.4\textwidth}
\centering
\includegraphics[width=8cm]{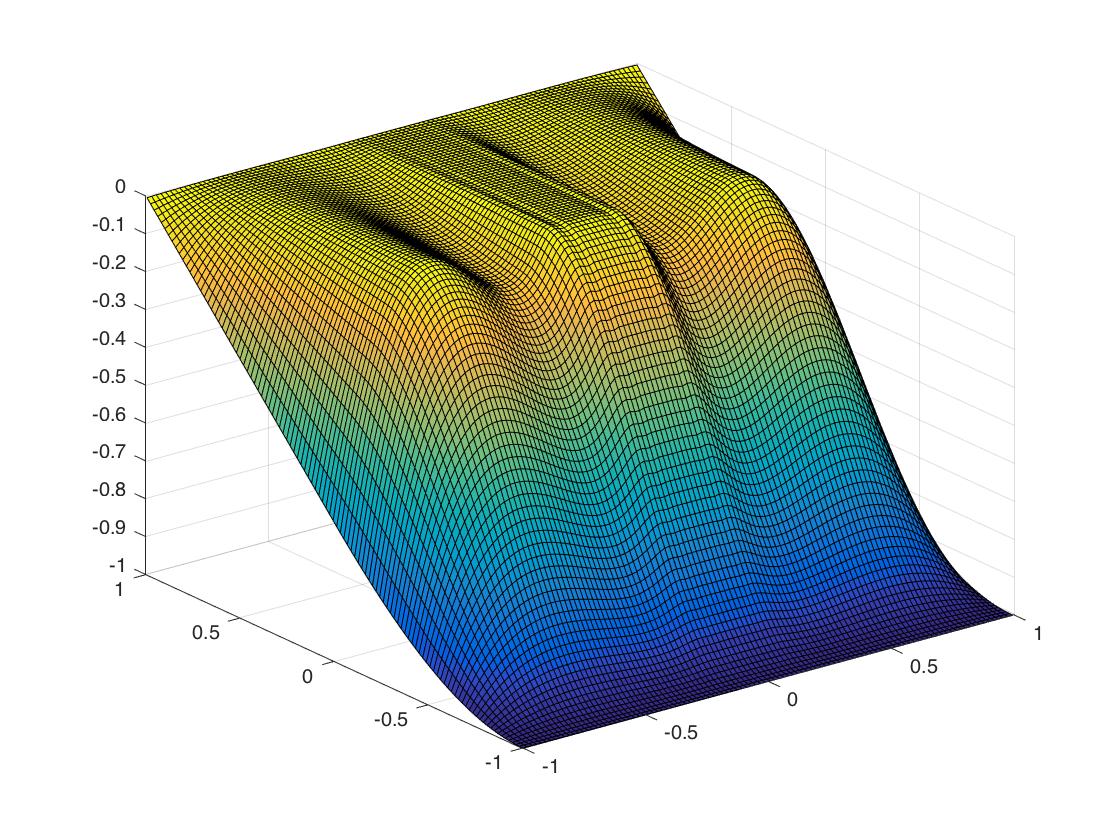}
\caption{
Global function devoid of a strictly decreasing path from $(0,1/2)$ to a global minimizer
}
	 \label{fig:nopath}
\end{subfigure}
\begin{subfigure}{.9\textwidth}
\centering
\includegraphics[width=8cm]{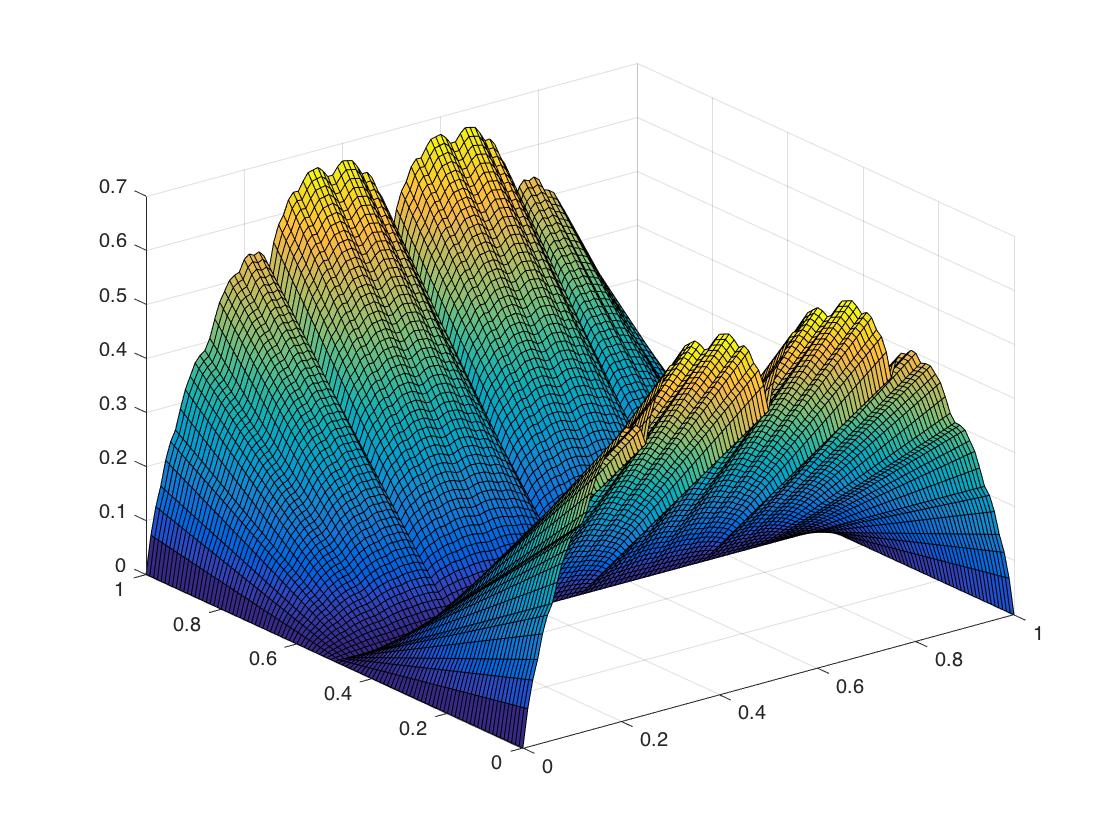}
\caption{Global function that is nowhere differentiable}
	 \label{fig:nodiff}
\end{subfigure}
\caption{Notable examples (with $x_1$-axis on the right and $x_2$-axis on the left)}
\label{fig:counter}
\end{figure}


\subsection{Proof of Proposition \ref{prop:composition}}
\label{subsec:composition}
($ \Longrightarrow $) Let $x \in S$ denote a local minimum of $\phi \circ f$. There exists $\epsilon > 0$ such that $\phi(f(x)) \leqslant \phi(f(y))$ for all $y \in S \setminus \{x\}$ with $\|x - y\|_2 \leqslant \epsilon$. Since $\phi$ is increasing, it holds that $f(x) \leqslant f(y)$. Since $f$ is global, we deduce that $x$ is a global minimum of $f$, that is to say $f(x) \leqslant f(y)$ for all $y \in S \setminus \{x\}$. Since $\phi$ is increasing, it holds that $\phi(f(x)) \leqslant \phi(f(y)) $ for all $y \in S \setminus \{x\}$. We conclude that $x$ is a global minimum of $\phi \circ f$.

($ \Longleftarrow $) Simply apply the previous argument to $\phi^{-1} \circ ( \phi \circ f)$, where $\phi^{-1}$ denotes the inverse of $\phi: f(S) \longrightarrow \phi \circ f(S)$.

\subsection{Proof of Proposition \ref{prop:change}}
\label{subsec:change}
($\Longrightarrow$) Let $x' \in S'$ denote a local minimum of $f \circ \varphi^{-1}$. There exists $\epsilon' > 0$ such that $f(\varphi^{-1}(x')) \leqslant f(\varphi^{-1}(y'))$ for all $y' \in S' \setminus \{x'\}$ with $\|x' - y'\|_2 \leqslant \epsilon'$. Since $\varphi$ is continuous, there exists $\epsilon>0$ such that $f(\varphi^{-1}(x')) \leqslant f(y)$ for all $y \in S \setminus \{\varphi^{-1}(x')\}$ with $\|\varphi^{-1}(x') - y\|_2 \leqslant \epsilon$. Hence, $\varphi^{-1}(x')$ is a local minimum of $f$. Since $f$ is global, it holds that $f(\varphi^{-1}(x')) \leqslant f(y)$ for all $y \in S$. Since $\varphi$ is a bijection, $f(\varphi^{-1}(x')) \leqslant f(\varphi^{-1}(y'))$ for all $y' \in S'$, implying that $x'$ is a global minimum of $f \circ \varphi^{-1}$. 

($ \Longleftarrow $) Simply apply the previous argument to $(f \circ \varphi^{-1}) \circ \varphi$.

\subsection{Proof of Proposition \ref{prop:definition}}
\label{subsec:definition}

One direction is obvious. For the other direction, we propose a proof by contrapositive. Let $X \subset S$ denote a local minimum that is not a global minimum. There exists $\epsilon > 0$ such that the uniform neighborhood $V := \{ y \in S ~|~ \exists\, x \in X: \|x-y\|_2 \leqslant \epsilon \}$ satisfies $f(x) \leqslant f(y)$ for all $x \in X$ and for all $y \in V \setminus X$. Also, there exists $z \in S \setminus V$ such that $f(z) < f(x)$ for all $x \in V$. Since $f$ is continuous on the compact set $V$, it attains a minimum $x' \in V$ such that $f(z) < f(x')$. If $x' \in X$, then for all $y \in S$ such that $\|x'-y\|_2 \leqslant \epsilon$, it holds that $f(z) < f(x') \leqslant f(y)$. Thus, $x'$ is local minimum that is not a global minimum. If $x' \in V\setminus X$, then $f(x') \leqslant f(x) \leqslant f(x')$ for all $x \in X$. Consider a point $x \in X$. For all $y \in S$ such that $\|x-y\|_2 \leqslant \epsilon$, it holds that $f(x) = f(y)$ if $y \in X$ and $f(x) \leqslant f(y)$ if $y \notin X$. Together with the fact that $f(z) < f(x') = f(x)$, we deduce that $x$ is a local minimum that is not a global minimum.

\subsection{Proof of Proposition \ref{prop:weakly global}}
\label{subsec:weakly global}
We propose a proof by contrapositive. Assume that $f$ is not constant on a local minimum $X \subset S$ that is not a global minimum. The minimum $X$ admits a uniform neighborhood $V$ such that $f(x) \leqslant f(y)$ for all $x \in X$ and for all $y \in V \setminus X$. Since $f$ is continuous on the compact set $V$, there exists $x' \in V$ such that $f(x') \leqslant f(x)$ for all $x \in V$. If $x' \in V \setminus \text{int}(X)$ where ``$\text{int}$'' stands for interior, then $f$ is constant on $X$ because $X$ is a local minimum. 
Therefore, $x' \in \text{int}(X)$ and $f(x') < f(x)$ for all $x \in \partial X := X \setminus \text{int}(X)$. Consider the compact set defined by $X' := \{ x \in X ~|~ f(x') = f(x) \}$. The set $V$ satisfies $f(x) < f(y)$ for all $x \in X'$ and $y \in V \setminus X'$. Since $X' \subset X$, there exists a uniform neighborhood $V'$ of $X'$ satisfying $f(x) < f(y)$ for all $x \in X'$ and for all $y \in V' \setminus X'$. Hence, $X'$ is a strict local minimum that is not global. To conclude, $f$ is not a weakly global function.

\subsection{Proof of Proposition \ref{prop:compact}}
\label{subsec:compact proof}
Consider a sequence of global functions $f_k$ that converge compactly towards $f$. Since $S \subset \mathbb{R}^n$ and $\mathbb{R}^n$ is a compactly generated space, it follows that $f$ is continuous. 
We proceed to prove that $f$ is a weakly global function by contradiction.
Suppose $X \subset S$ is a strict local minimum that is not global minimum. There exists $\epsilon > 0$ such that the uniform neighborhood $V := \{ y \in S ~|~ \exists\, x \in X: \|x-y\|_2 \leqslant \epsilon \}$ satisfies $f(x) < f(y)$ for all $x \in X$ and for all $y \in V \setminus X$. Since $f$ is continuous on the compact set $X$, it attains a minimal value on it, say $\inf_X f := \alpha + \inf_{S} f$ where $\alpha > 0$ since $X$ is not a global minimum. Consider a compact set $V \subset K \subset S$ such that $\inf_{K} f \leqslant \alpha/2 + \inf_{S} f$. Since $f$ is continuous on the compact set $\partial V$, it attains a minimal value on it, say $\inf_{\partial V} f := \beta + \inf_{X} f$ where $\beta > 0$ by strict optimality. Let $\gamma := \min \{ \alpha , \beta \}$. For a sufficiently large value of $k$, compact convergence implies that $|f_k(y)-f(y)| \leqslant \gamma /3$ for all $y \in K$. Since the function $f_k$ is compact on $V$, it attains a minimum, say $z \in V$. Therefore,
\begin{equation}
f_k(z) ~\leqslant~ \gamma/3 + \inf_{V} f ~\leqslant~ \beta/3 + \inf_{V} f ~<~ 2\beta/3 + \inf_{V} f
\end{equation}
\begin{equation}
\leqslant~ - \gamma/3 + \beta + \inf_V f ~\leqslant~ - \gamma/3 + \inf_{\partial V} f ~\leqslant~ \inf_{\partial V} f_k.
\end{equation}
Thus, $z \in \text{int}(V)$. We now proceed to show by contradiction that $z$ is a local minimum of $f_k$. Assume that for all $\epsilon' > 0$, there exists $y' \in S \setminus \{x\}$ satisfying $\|x - y' \|_2 \leqslant \epsilon'$ such that $f_k(z) > f_k(y')$. We can choose $\epsilon'$ small enough to guarantee that $y'$ belongs to $V$ since $z \in \text{int}(V)$. The point $y'$ then contradicts the minimality of $z$ on $V$. This means that $z \in V$ is a local minimum of $f_k$. Now, observe that 
\begin{equation}
\inf_K f_k ~\leqslant~ \gamma/3 + \inf_K f ~\leqslant~ \gamma/3 + \alpha/2 + \inf_S f ~\leqslant~ 2\alpha/3 + \inf_S f ~<~ 5\alpha/6 + \inf_S f 
\end{equation}
\begin{equation}
\leqslant~ \alpha - \gamma/3 + \inf_S f ~=~ - \gamma/3 + \inf_X f ~=~ -\gamma/3 + \inf_V f ~\leqslant~ \inf_V f_k ~\leqslant~ f_k(z).
\end{equation} 
Thus, $z$ is not a global minimum of $f_k$. This contradicts the fact that $f_k$ is a global function.

\subsection{Proof of Lemma \ref{lemma:clarke}}
\label{subsec:clarke}
Based on the Clarke derivative \cite{clarke1975,clarke1990} for locally Lipschitz functions, the first-order optimality condition reads
\begin{equation}
0 \in \sum\limits_{i_1,\hdots,i_{d-1}=1}^n x_{i_1} \hdots x_{i_{d-1}} ~ \text{sign}(x_{i_1} \hdots x_{i_{d-1}} x_i - y_{i_1} \hdots y_{i_d} y_i) ~~~,~~~ i = 1,\hdots,n
\end{equation}
where
\begin{equation}
\text{sign}(x) :=
\left\{
\begin{array}{cl}
-1 & \text{if} ~ x < 0, \\
\big[-1,1\big] & \text{if} ~ x = 0, \\
1 & \text{if} ~ x > 0.
\end{array}
\right.
\end{equation}
If $y_i = 0$ for some $i \in \{1,\ldots,n\}$, then the above equations readily yield 
\begin{equation}
0 \in \sum\limits_{i_1,\hdots,i_{d-1}=1}^n x_{i_1} \hdots x_{i_{d-1}} ~ \text{sign}(x_{i_1} \hdots x_{i_{d-1}} x_i)
=
\text{sign}( x_i)
 \sum\limits_{i_1,\hdots,i_{d-1}=1}^n |x_{i_1} \hdots x_{i_{d-1}}|
\end{equation}
which implies
$x_i = 0$. This reduces the dimension of the problem from $n$ to $n-1$, so without loss of generality we may assume that $y_i \neq 0$ for all $i =1,\hdots,n$. After a division, observe that the following inclusion is satisfied:
\begin{equation}
\label{eq:root}
0 \in \sum\limits_{i_1,\hdots,i_{d-1}=1}^n |y_{i_1} \hdots y_{i_{d-1}}| \frac{x_{i_1} \hdots x_{i_{d-1}}}{y_{i_1} \hdots y_{i_{d-1}}} \text{sign} \left( \frac{x_{i_1} \hdots x_{i_{d-1}}}{y_{i_1} \hdots y_{i_{d-1}}} t - 1 \right)
\end{equation}
for all $t \in \{ x_1/y_1, \hdots, x_n / y_n\}$. Each term with $x_{i_1} \hdots x_{i_{d-1}} \neq 0$ in the above sum is an increasing step (set-valued) function of $t \in \mathbb{R}$ since it is the Clarke derivative of the convex function 
\begin{equation}
g(t)= | x_{i_1} \hdots x_{i_{d-1}} t - y_{i_1} \hdots y_{i_{d-1}} |.
\end{equation}
The above sum is thus a increasing step function of $t \in \mathbb{R}$. Hence, the roots $x_1/y_1, \hdots, x_n / y_n$ all along belong to the same step. Jumps between the steps occur exactly at the following set of points:
\begin{equation}
\left\{ \left. \frac{y_{i_1} \hdots y_{i_{d-1}}}{x_{i_1} \hdots x_{i_{d-1}}} ~~~ \right| ~~~i_1,\hdots,i_{d-1} \in \{1,\ldots, n\} ~~\text{and}~~ x_{i_1} \hdots x_{i_{d-1}} \neq 0 \right\}
\end{equation}
This set is empty when $x=0$; otherwise, none of its elements are equal to zero because $y \neq 0$. Given a jump point $\alpha \neq 0$ in the above set, the roots must therefore be all before or all after, that is to say:
\begin{equation}
\frac{x_1}{y_1}, \hdots, \frac{x_n} {y_n} \leqslant \alpha ~~~~~~~~~\text{or}~~~~~~~~~ \alpha \leqslant \frac{x_1}{y_1}, \hdots, \frac{x_n} {y_n}.
\end{equation}
We next prove that 
\begin{equation}
\label{eq:implications}
\alpha > 0 ~\Longrightarrow~ \frac{x_1}{y_1}, \hdots, \frac{x_n} {y_n} \leqslant \alpha ~~~~~~~~~\text{and}~~~~~~~~~ \alpha < 0 ~\Longrightarrow~ \alpha \leqslant \frac{x_1}{y_1}, \hdots, \frac{x_n} {y_n}.
\end{equation}
Let us prove the first implication by contradiction. Assume that there exists $ k \in \{1,\ldots,n\}$ such that $\alpha < x_k/y_k$. Since one root is after the jump point $\alpha$, all other roots are after the jump point $\alpha$. In particular, for all $i \in \{1,\ldots,n\}$, we have
\begin{equation}
0 < \alpha :=\frac{y_{i_1} \hdots y_{i_{d-1}}}{x_{i_1} \hdots x_{i_{d-1}}} \leqslant \frac{x_{i}}{y_{i}}
\end{equation}
Therefore, all the roots are positive. By multiplying the above equation by the positive number $\frac{x_{i_1}y_{i}}{y_{i_1}x_i}$, we obtain
\begin{equation}
\label{eq:before_error}
\frac{y_{i} y_{i_2} \hdots y_{i_{d-1}}}{x_{i} x_{i_2} \hdots x_{i_{d-1}}} \leqslant \frac{x_{i_1}}{y_{i_1}}.
\end{equation}
Note that the left-hand side is a jump point, and the right-hand side is a root. Therefore, all the roots are after \textbf{(This is the error in the proof: a root must be strictly after a jump point so that all roots are after the jump point)}, and in particular:
\begin{equation}
\label{eq:after_error}
\frac{y_{i} y_{i_2} \hdots y_{i_{d-1}}}{x_{i} x_{i_2} \hdots x_{i_{d-1}}} \leqslant \frac{x_{i}}{y_{i}}.
\end{equation}
Again, since the roots are positive, by multiplying by $\frac{x_{i_2}y_{i}}{y_{i_2}x_i}$, we get
\begin{equation}
\frac{y_{i}^2 y_{i_3} \hdots y_{i_{d-1}}}{x_{i}^2 x_{i_3} \hdots x_{i_{d-1}}} \leqslant \frac{x_{i_2}}{y_{i_2}}.
\end{equation}
Similarly, the left-hand side is a jump point, and the right-hand side is a root. Thus, all the roots are after, and in particular:
\begin{equation}
\frac{y_{i}^2 y_{i_3} \hdots y_{i_{d-1}}}{x_{i}^2 x_{i_3} \hdots x_{i_{d-1}}} \leqslant \frac{x_{i}}{y_{i}}.
\end{equation}
Continuing this process, we ultimately obtain that 
\begin{equation}
\frac{y_i^{d-1}}{x_i^{d-1}} \leqslant \frac{x_i}{y_i}
\end{equation}
that is to say $1 \leqslant x_i / y_i$. If the inequality is an equality for all $i \in \{1,\ldots,n\}$, then $\alpha = 1 = x_k / y_k$, which is impossible since $\alpha < x_k / y_k$. Thus, there exists one root $t$ of \eqref{eq:root} that is strictly greater than one. But this implies that every term in the sum in \eqref{eq:root} is strictly positive, which is impossible. 
As a result, the first implication in \eqref{eq:implications} is true. 

We next prove the second implication in \eqref{eq:implications} by contradiction. Assume that there exists $ k \in \{1,\ldots,n\}$ such that $x_k/y_k < \alpha$. Since one root is before the jump point $\alpha$, all other roots are before the jump point $\alpha$.  In particular, for all $i \in \{1,\ldots,n\}$, we have 
\begin{equation}
\frac{x_{i}}{y_{i}} \leqslant  \frac{y_{i_1} \hdots y_{i_{d-1}}}{x_{i_1} \hdots x_{i_{d-1}}} := \alpha<0.
\end{equation}
Therefore, all the roots are negative. Since $\alpha < 0$ is the product of $d-1$ negative terms, it must be that $d$ is even.
Observe that $\frac{x_{i_1}y_{i}}{y_{i_1}x_i} > 0$ because it is a ratio of two roots.
Now, similar to the case $\alpha >0$, we obtain
\begin{equation}
\frac{x_{i_1}}{y_{i_1}} \leqslant \frac{y_{i} y_{i_2} \hdots y_{i_{d-1}}}{x_{i} x_{i_2} \hdots x_{i_{d-1}}}.
\end{equation}
The right-hand side is a jump point, and the left-hand side is a root. Thus, all the roots are before, and in particular:
\begin{equation}
\frac{x_{i}}{y_{i}} \leqslant \frac{y_{i} y_{i_2} \hdots y_{i_{d-1}}}{x_{i} x_{i_2} \hdots x_{i_{d-1}}}.
\end{equation}
Continuing this process (as in the case where $\alpha >0$), we ultimately obtain that 
\begin{equation}
\frac{x_i}{y_i} \leqslant 
\frac{y_i^{d-1}}{x_i^{d-1}} .
\end{equation}
Since $d$ is even and $x_i / y_i < 0$, this implies that $x_i / y_i \leqslant -1$. If the inequality is an equality for all $i \in \{1,\ldots,n\}$, then $\alpha = -1 = x_k / y_k$, which is impossible since $x_k / y_k < \alpha$. Thus, there exists one root $t$ of \eqref{eq:root} that is strictly less than $-1$. But this implies that every term in the sum in \eqref{eq:root} is strictly negative, which is impossible. Consequently, the second implication in \eqref{eq:implications} holds. 

Let us apply \eqref{eq:implications} to a root $x_{i_d}/y_{i_d}$ for some $i_d \in \{1,\ldots,n\}$: 
$$
\frac{y_{i_1} \hdots y_{i_{d-1}}}{x_{i_1} \hdots x_{i_{d-1}}} > 0 ~\Longrightarrow~ \frac{x_{i_d}}{y_{i_d}} \leqslant \frac{y_{i_1} \hdots y_{i_{d-1}}}{x_{i_1} \hdots x_{i_{d-1}}} ~~~~~~~~~\text{and}~~~~~~~~~ \frac{y_{i_1} \hdots y_{i_{d-1}}}{x_{i_1} \hdots x_{i_{d-1}}} < 0 ~\Longrightarrow~ \frac{y_{i_1} \hdots y_{i_{d-1}}}{x_{i_1} \hdots x_{i_{d-1}}} \leqslant \frac{x_{i_d}}{y_{i_d}}.
$$
In both cases we find that 
\begin{equation}
\frac{x_{i_1} \hdots x_{i_{d}}}{y_{i_1} \hdots y_{i_{d}}} \leqslant 1.
\label{eq:lessthanone}
\end{equation}
This inequality holds for all jump points (i.e. for all indices $i_1,\hdots,i_{d-1} \in \{1,\ldots, n\}$ such that $x_{i_1} \hdots x_{i_{d-1}} \neq 0$) and it is trivially true for all indices such that $x_{i_1} \hdots x_{i_{d-1}} = 0$. Therefore, \eqref{eq:lessthanone} is true for all $i_1,\hdots,i_{d} \in \{1,\ldots, n\}$, which completes the proof of this lemma. 

\subsection{Proof of Lemma \ref{lemma:box}}
\label{subsec:box}
When $x \in S$, notice that
\begin{equation}
\begin{array}{rcl}
f_1(x) & = & \sum\limits_{i_1,\hdots,i_d=1}^n |x_{i_1} \hdots x_{i_d} - y_{i_1} \hdots y_{i_d}| \\\\
& = & \sum\limits_{i_1,\hdots,i_d=1}^n |y_{i_1} \hdots y_{i_d}| \left|\frac{x_{i_1} \hdots x_{i_{d}}}{y_{i_1} \hdots y_{i_{d}}} - 1\right| \\\\
& = & \sum\limits_{i_1,\hdots,i_d=1}^n |y_{i_1} \hdots y_{i_d}| - |y_{i_1} \hdots y_{i_d}| \frac{x_{i_1} \hdots x_{i_{d}}}{y_{i_1} \hdots y_{i_{d}}}  \\\\
& = &  \left( \sum\limits_{i=1}^n |y_i | \right)^d - \left( \sum\limits_{i=1}^n |y_i | \frac{x_i}{y_i} \right)^d.
\end{array}
\end{equation}

Given $\alpha>0$, consider the function 
$\phi_\alpha:  f_1(S)  \longrightarrow  \mathbb{R}$
defined by
\begin{equation}
\phi_\alpha(t) = -\Bigg[ -t - \left( \sum\limits_{i=1}^n |y_i | \right)^d \Bigg]^\alpha.
\end{equation}
If $d$ is odd, then $\phi_\alpha$ is increasing when taking $\alpha = 1/d$. If $d$ is even, then $\phi_\alpha$ is increasing when $-t - ( \sum_{i=1}^n |y_i |)$ is positive and $\alpha = 2/d$.
Next, define the set 
\begin{equation}
S' := \left\{ ~ x \in \mathbb{R}^n ~ \left| ~ x_{i_1} \hdots x_{i_d} \leqslant 1~, ~~~ \forall ~ i_1,\hdots,i_{d} \in \{1,\ldots, n\} \right. ~ \right\}
\end{equation}
and consider the homeomorphism 
$\varphi:   S  \longrightarrow  S'$
defined by
\begin{equation}
\varphi(x) = \left( \frac{x_1}{y_1} , \hdots , \frac{x_n}{y_n} \right).
\end{equation}
According to Proposition \ref{prop:composition} and Proposition \ref{prop:change}, $f$ is a global function on $S$, i.e. $f_1 \in \mathcal{G}(S)$, if and only if $\phi_\alpha \circ f_1 \circ \varphi^{-1}$ is a global function on $S'$.
Thus, when $d$ is odd, $f_1 \in \mathcal{G}(S)$ if and only if
$f_{\text{odd}}(x) := \phi_{1/d} \circ f_1 \circ \varphi^{-1}(x)=- \sum_{i=1}^n |y_i| x_i \in \mathcal{G}(S')$.
When $d$ is even, $f_1 \in \mathcal{G}(S)$ if and only if
$f_{\text{even}}(x) := \phi_{2/d} \circ f_1 \circ \varphi^{-1}(x)=- \left( \sum_{i=1}^n |y_i| x_i \right)^2 \in \mathcal{G}(S')$.

Consider the case when $d$ is odd. 
For all $i_1,\hdots,i_{d} \in \{1,\ldots, n\}$, define the constraint function $g_{ i_1 , \hdots i_d}(x) := x_{i_1} \hdots x_{i_{d}} - 1$. If $x_1 \hdots x_n \neq 0$, then for any $i_1,\hdots,i_{d} \in \{1,\ldots, n\}$, it satisfies
\begin{equation}
\nabla ~ g_{ i_1 , \hdots i_d}(x) = 
\begin{pmatrix}
N(1,i_1 , \hdots i_d)/x_1 \\
\vdots \\
N(n,i_1 , \hdots i_d)/x_n
\end{pmatrix} 
x_{i_1} \hdots x_{i_{d}}
\end{equation}
where $\nabla ~ g_{ i_1 , \hdots i_d}(x)$ denotes the gradient of $g_{ i_1 , \hdots i_d}$ at $x$ and
$N(i,i_1,\hdots,i_d)$ denotes the number of indices among $i_1,\hdots,i_d$ that are equal to $i$. If the constraint $g_{ i_1 , \hdots i_d}(x) \leqslant 0$ is active, then
\begin{equation}
-x^T ~ \nabla  g_{ i_1 , \hdots i_d}(x) = - \sum\limits_{k=1}^n N(k,i_1 , \hdots i_d)  < 0.
\end{equation}
The Mangasarian-Fromovitz constraint qualification thus holds. 
A local minimum $x \in \mathbb{R}^n$ for the problem $\inf_{x\in S'} f_{\text{odd}}(x)$ must therefore satisfy the Karush-Kuhn-Tucker conditions:
\begin{equation}
\left\{
\begin{array}{l}
\sum\limits_{\tiny \begin{array}{c}i_1,\hdots,i_d=1 \\ N(i,i_1,\hdots,i_d) \neq 0 \end{array}}^n \lambda_{i_1,\hdots,i_d} N(i,i_1,\hdots,i_d) \frac{x_{i_1} \hdots x_{i_{d}}}{x_i} = 0 ~,~~~ \forall ~ i \in \{1,\ldots,n\}~, \\\\
x_{i_1} \hdots x_{i_{d}} \leqslant 1 ~,~~~ \forall ~ i_1,\hdots,i_{d} \in \{1,\ldots, n\}~,\\\\
\lambda_{i_1,\hdots,i_d} \geqslant 0~, \\\\
\lambda_{i_1,\hdots,i_d} ( x_{i_1} \hdots x_{i_{d}} - 1 ) = 0~.
\end{array}
\right.
\end{equation}
Here $\lambda_{i_1,\hdots,i_d} \geq 0, i_1,\hdots,i_{d} \in \{1,\ldots, n\},$ are the Lagrange multipliers.
If $x_i \neq 0$ for some $i \in \{1,\ldots,n\}$, then, by complementarity slackness, the first line yields
\begin{equation}
0 ~~~ < ~~~ |y_i| ~~~= \frac{1}{x_i} \sum\limits_{\tiny \begin{array}{c}\lambda_{i_1,\hdots,i_d}>0 \\ N(i,i_1,\hdots,i_d) \neq 0 \end{array}} \underbrace{\lambda_{i_1,\hdots,i_d} N(i,i_1,\hdots,i_d)}_{>0}
\end{equation}
which implies that $x_i > 0$. As a result, $x \geqslant 0$. Together with feasibility, it results that $0 \leqslant x_i^d \leqslant 1$, leading to the inequalities $0 \leqslant x_i \leqslant 1$ for all $i \in \{1,\ldots,n\}$.
Following Proposition \ref{prop:subset}, $f_{\text{odd}}$ is thus global on $S'$ if $f_{\text{odd}} \in \mathcal{G}(S'')$
where
\begin{equation}
S'' := \left\{ ~ x \in \mathbb{R}^n ~ \left| ~ 0 \leqslant x_i \leqslant 1~, ~~~ \forall ~ i \in \{1,\ldots, n\} \right. ~ \right\}.
\end{equation}
From the notion of global functions, $f_{\text{odd}} \in \mathcal{G}(S'')$ if the problem
\begin{equation}
\inf_{x\in \mathbb{R}^n} ~~~~~~ - \sum_{i=1}^n |y_i| x_i ~~~~~~~\text{subject to}~~~~~~ 0 \leqslant x_i \leqslant 1 ~,~~~~~~~~  \forall ~ i \in \{1,\ldots, n\}
\end{equation}
has no spurious local minima, which is obvious because the problem is a linear program.

Consider the case when $d$ is even. Since a feasible point $x \in S'$ satisfies $x_i^d \leqslant 1$, it must be that $-1 \leqslant x_i \leqslant 1$. Conversely, any point such that $-1 \leqslant x_i \leqslant 1$ belongs to $S'$. 
This implies that 
\begin{equation}
S' := \left\{ ~ x \in \mathbb{R}^n ~ \left| ~ -1 \leqslant x_i \leqslant 1~, ~~~ \forall ~ i \in \{1,\ldots, n\} \right. ~ \right\}.
\end{equation}
According to Proposition \ref{prop:cover}, $f_{\text{even}}(x) \in \mathcal{G}(S')$ if
$f_{\text{even}}(x)$ is a global function on both sets
$S' \cap \{x \in \mathbb{R}^n| \sum_{i=1}^n |y_i| x_i \geq 0\}$ and $S' \cap \{x \in \mathbb{R}^n| \sum_{i=1}^n |y_i| x_i \leq 0\}$, and $f_{\text{even}}(x)$ takes the same optimal value on both sets (the latter is obvious using symmetry).
Using Proposition \ref{prop:composition} again, we find that $f_{\text{even}}(x)$ is a global function on these two sets if and only if
\begin{equation}
-\sum_{i=1}^n |y_i| x_i \in \mathcal{G}\left(S' \cap \left\{x \in \mathbb{R}^n| \sum_{i=1}^n |y_i| x_i \geq 0\right\}\right)
\end{equation}
and 
\begin{equation}
\sum_{i=1}^n |y_i| x_i \in \mathcal{G}\left(S' \cap \left\{x \in \mathbb{R}^n| \sum_{i=1}^n |y_i| x_i \leq 0  \right\}\right),
\end{equation}
which are true because they are associated with the following linear programs:
\begin{equation}
\inf_{x\in \mathbb{R}^n} ~~~ - \sum_{i=1}^n |y_i| x_i ~~~~~~\text{subject to}~~~~~~ \left\{
\begin{array}{l}
-1 \leqslant x_i \leqslant 1 ~,~~~~~~~~  \forall ~ i \in \{1,\ldots, n\} \\\\
\sum\limits_{i=1}^n |y_i| x_i \geqslant 0.
\end{array}
\right.
\end{equation}
and 
\begin{equation}
\inf_{x\in \mathbb{R}^n} ~~~ \hphantom{-} \sum_{i=1}^n |y_i| x_i ~~~~~~\text{subject to}~~~~~~ \left\{
\begin{array}{l}
-1 \leqslant x_i \leqslant 1 ~,~~~~~~~~  \forall ~ i \in \{1,\ldots, n\} \\\\
\sum\limits_{i=1}^n |y_i| x_i \leqslant 0.
\end{array}
\right.
\end{equation}

\end{document}